\documentclass{amsart}

 \usepackage{amssymb}
\usepackage{amsthm}
\usepackage{fullpage}
\newtheorem{defi}{Definition}[section]
\newtheorem{theorem}{Theorem}[section]
\newtheorem{lemma}{Lemma}[section]
\newtheorem{coro}{Corollary}[section]
\newtheorem{remark}{Remark}[section]
\numberwithin{equation}{section}
\newtheorem*{theorem*}{Theorem}

\def\O{\Omega}
\def\o{\omega}
\def\R{{\mathbb R}}

\def\g{\gamma}

\def\Xint#1{\mathchoice
  {\XXint\displaystyle\textstyle{#1}}%
  {\XXint\textstyle\scriptstyle{#1}}%
  {\XXint\scriptstyle\scriptscriptstyle{#1}}%
  {\XXint\scriptscriptstyle\scriptscriptstyle{#1}}%
  \!\int}
\def\XXint#1#2#3{{\setbox0=\hbox{$#1{#2#3}{\int}$}
    \vcenter{\hbox{$#2#3$}}\kern-.5\wd0}}

\def\avgint{\Xint-}
\newcommand{\vertiii}[1]{{\left\vert\kern-0.25ex\left\vert\kern-0.25ex\left\vert #1 
    \right\vert\kern-0.25ex\right\vert\kern-0.25ex\right\vert}}

\begin{document}

\subjclass[2020]{Primary: 46E35; Secondary: 46B70, 26D10.}

\title[An interpolation result for $A_1$ weights]{An interpolation result for $A_1$ weights with applications to fractional Poincar\'e inequalities}

\author{Irene Drelichman}
\address{CMaLP, Departamento de Matem\'atica, Facultad de Ciencias Exactas, Universidad Nacional de La Plata, Argentina}
\email{irene@drelichman.com}

\thanks{Supported by FONCYT under grant PICT-2018-03017, and by Universidad de Buenos Aires under grant 20020190100273BA. The author is a member of CONICET, Argentina.}

\begin{abstract}
We characterize the real interpolation space between weighted $L^1$ and $W^{1,1}$  spaces on arbitrary  domains different from $\R^n$, when the weights are positive powers of the distance to the boundary multiplied by an $A_1$ weight. As an  application of this result we obtain weighted fractional Poincar\'e inequalities with sharp dependence on the fractional parameter $s$ (for $s$ close to 1) and show that they are equivalent to a weighted Poincar\'e inequality for the gradient.
\end{abstract}

\keywords{Fractional Sobolev spaces, Gagliardo seminorm, irregular domains, interpolation spaces, Muckenhoupt weights.}

\maketitle

\section{Introduction}
Given a domain  $\Omega \subset\R^n, \O \neq \R^n$, we denote by $d(x)=d(x, \partial \O)$ the distance from $x$ to the boundary. Let $\alpha, \beta \ge 0$, and $\o$ be a weight in  Muckenhoupt's class $A_1$, that is, such that  $M\omega(x) \le C \omega (x)$ a.e., where $M$ is the Hardy-Littlewood maximal function. We consider the weighted Sobolev space
$$
W^{1,1}_\o(\O, d^\alpha, d^\beta)=\{f \in L^1_\o(\O,  d^\alpha) : \|\nabla f\|_{L^1_\o(\O, d^\beta)} < \infty \}
$$
where $\|f\|_{L^1_\o(\O, d^\beta)}=\|f  \o d^\beta\|_{L^1(\O)}$.

The first goal of this paper is to  show that, for any such domain, and any $\alpha\ge 0$, one has 
\begin{equation}
\label{interpolado}
(L^1_\o(\O, d^\alpha), W^{1,1}_{\o} (\O, d^\alpha, d^{\alpha+1}))_{s,1}= \widetilde W^{s,1}_{\o }(\O, d^\alpha, d^{\alpha+s})
\end{equation}
 with equivalence of norms, where 
 $$
 \widetilde W^{s,1}_{\o}(\O, d^\alpha, d^\beta) = \{ f\in L^1_\o(\O, d^\alpha) : |f|_{\widetilde W^{s,1}_{\o}(\O, d^\beta)} <\infty \}
$$
and
$$
|f|_{\widetilde W^{s,1}_{\o}(\O, d^\beta)} =  \int_\O \int_{|x-y|<\tau d(x)} \frac{|f(x)-f(y)|}{|x-y|^{n+s}} \, dy \, d(x)^\beta \o(x) \, dx.
$$
 
This result generalizes the one  in  \cite[Theorem 1.1]{ADD}, which corresponds to the case   $\o\equiv 1$ (notice that that result is written for bounded domains, but the same arguments apply as long as $\O\neq \R^n$). The proof of the embedding  $\widetilde W^{s,1}_{\omega}(\Omega, d^\alpha, d^{\alpha+s}) \subseteq (L^1_\omega(\Omega, d^\alpha), W^{1,1}_{\omega}(\Omega, d^\alpha, d^{\alpha+1}))_{s,1} $  follows closely the one in that paper, modifying it to include the $A_1$ weight. But, because of the presence of the weight,  the  opposite embedding requires a completely different proof. We borrow  some ideas from \cite{GKS}, but we adapt them to our seminorm and to the presence of different powers of the distance to the boundary. We remark that,  among other differences, in  \cite{GKS}  both the function and its (generalized) gradient belong to the same weighted space, which is not our case.

The characterization in \eqref{interpolado} is strongly related to the obtention of fractional Poincar\'e inequalities with sharp dependence on the fractional parameter $s$, for $s$ close to 1. 

Recall that,  for a cube $Q$, $1\le p<\frac{1}{s}$, and $\frac12 \le s<1$, it was proved in \cite[Theorem 1]{BBM2} that 
\begin{equation}
\label{fraccionaria-BBM}
\|f - f_Q\|^p_{L^p(Q)} \lesssim \frac{(1-s)}{(n-sp)^{p-1}} \int_Q \int_Q \frac{|f(x)-f(y)|^p}{|x-y|^{n+sp}} \, dx \, dy,
\end{equation}
where $f_Q$ stands for the integral average of $f$ over $Q$. Here, the implicit constant depends on the side-length of $Q$ but, in what follows, we will not be interested in such dependence. Other proofs and extensions of this inequality can be found in \cite{HMPV, MPW, MS, P}.

For irregular domains, a more suitable fractional norm was introduced in \cite{HV}, and it was shown that for any bounded John domain $\O\subset \R^n$ (see definition below)  and any fixed constant $0<\tau<1$,
\begin{equation}
\label{improved-fraccionaria}
\|f - f_\O\|_{L^p(\O)}^p \lesssim   \int_\O \int_{|x-y|<\tau d(x)} \frac{|f(x)-f(y)|^p}{|x-y|^{n+sp}} \, dy \, dx.
\end{equation}
Generalizations of this result can be found in \cite{DD-fracpoincare, DIV, G, LGO}, but it should be noted that the scaling factor $(1-s)$ in the right-hand side of \eqref{fraccionaria-BBM} cannot be obtained with any of those proofs.  This turns out to be a drawback, since this factor plays a key role in the limiting behavior of the seminorm when $s\to 1^-$, and it relates fractional and classical Poincar\'e inequalities. Indeed, it was proved in \cite{BBM}  (see also \cite{M}) that, for a bounded extension domain  $\O$,  $1\le p<\infty$,
 and $f\in W^{1,p}(\O)$,
$$
\lim_{s\to 1^-} (1-s) \int_\O \int_\O \frac{|f(x)-f(y)|^p}{|x-y|^{n+sp}} \, dx \, dy = K_{n,p} \|\nabla f\|_{L^p(\O)}
$$
where $K_{n,p}$ is an explicit constant, so that one can recover from \eqref{fraccionaria-BBM} the classical Poincar\'e inequality for the gradient in $Q$.

For arbitrary bounded domains, the analogous result holds using the restricted fractional seminorm. Namely, it was proved in \cite{DD-BBM} that, for $f\in W^{1,p}(\O)$, $1< p<\infty$,
$$
\lim_{s\to 1^-} (1-s)  \int_\O \int_{|x-y|<\tau d(x)} \frac{|f(x)-f(y)|^p}{|x-y|^{n+sp}} \, dy \, dx = K_{n,p} \|\nabla f\|_{L^p(\O)},
$$
and this result was extended to $p=1$ in \cite{Mo}. This suggests that \eqref{improved-fraccionaria} should also hold with the $(1-s)$ factor. The second goal of this paper is to show that this is indeed the case when $p=1$, in the more general weighted setting. More precisely, we prove  that for bounded John domains one has
$$
\inf_{c\in\R} \| f-c\|_{L^1_\o(\O, d^\alpha)} \lesssim  \frac{(1-s)}{s(n+s)}  \int_\O \int_{|x-y|<\tau d(x)} \frac{|f(x)-f(y)|}{|x-y|^{n+s}} \, dy \, d(x)^{\alpha+s} \o(x) \, dx
$$
whenever $\o  \in A_1$. This is done by showing that this inequality is equivalent to a weighted Poincar\'e inequality for the gradient, which is known.  The proof of this equivalence uses some ideas from Oscar Dom\'inguez Bonilla, which relate bounds for the $K$-functional corresponding to \eqref{interpolado} to the obtention of sharp inequalities,  so the author would like to thank him for generously sharing them. It is worth noting that $K$-functionals have also been recently used in a different way to derive self-improving type inequalities of several classical inequalities in \cite{DLTYY}. 

Finally, the author also wishes to thank  the anonymous referee for carefully reading the manuscript and giving many valuable suggestions.

\section{Notation and Preliminary Results}

As usual, we will write $A\lesssim B$ to mean $A\le CX$ whenever $C$ is a positive constant independent of relevant quantities. Throughout this paper we shall only keep track of the dependence of the constants with respect to the interpolation parameter $s$, that we will use later in our arguments. 

Let $L(\O)$ denote the collection of measurable functions $f: \O \to \mathbb{R}$.  In what follows, we will consider the following weighted Lebesgue and Sobolev spaces
$$
L^1_\omega(\O, d^\alpha) = \{ f \in L(\O) : \|f\|_{L^1_\omega(\O, d^\alpha)}= \|f \o d^\alpha\|_{L^1(\O)} < \infty\}
$$
$$
W^{1,1}_\omega(\O, d^\alpha, d^{\alpha+1}) = \{ f \in L^1_\omega(\O, d^\alpha) : \|\nabla f\|_{L^1_\o(\O, d^{\alpha+1})} < \infty\}
$$
and their fractional counterparts
$$
 W^{s,1}_\omega(\O, d^\alpha, d^{\alpha+s}) = \{ f \in L^1_\omega(\O, d^\alpha) : |f|_{ W^{s,1}_{\o}(\O, d^{\alpha+s})} < \infty\}
$$
$$
\widetilde W^{s,1}_\omega(\O, d^\alpha, d^{\alpha+s}) = \{ f \in L^1_\omega(\O, d^\alpha) : |f|_{\widetilde W^{s,1}_{\o}(\O, d^{\alpha+s})} < \infty\}
$$
where
$$
|f|_{ W^{s,1}_{\o}(\O, d^{\alpha+s})} =  \int_\O \int_\O \frac{|f(x)-f(y)|}{|x-y|^{n+s}} \, dy \, d(x)^{\alpha+s} \o(x) \, dx,
$$
and
$$
|f|_{\widetilde W^{s,1}_{\o}(\O, d^{\alpha+s})} =  \int_\O \int_{|x-y|<\tau d(x)} \frac{|f(x)-f(y)|}{|x-y|^{n+s}} \, dy \, d(x)^{\alpha+s} \o(x) \, dx.
$$

By definition, for $0<s<1$, the real interpolation space between  $L^1_\o(\O, d^\alpha)$ and  $W^{1,1}_{\o} (\O, d^\alpha, d^{\alpha+1})$ is given by

$$
(L^1_\o(\O, d^{\alpha}), W^{1,1}_\o(\Omega, d^{\alpha }, d^{\alpha+1}))_{s,1} = \{ f\in L^1_\o(\Omega, d^{\alpha }) : \|f\|_{(L^1_\o(\Omega, d^{\alpha }), W^{1,1}_\o(\Omega, d^{\alpha }, d^{\alpha+1}))_{s,1}} <\infty\}
$$
with 
\begin{equation}\label{K}
\|f\|_{(L^1_\o(\Omega, d^{\alpha }), W^{1,1}_\o(\Omega, d^{\alpha }, d^{\alpha+1}))_{s,1}} = \int_0^\infty \lambda^{-s} K(f, \lambda) \,  \frac{d\lambda}{\lambda} 
\end{equation}
and  
\begin{equation}\label{K2}
K(f, \lambda)= \inf\{ \|g\|_{L^1_\o(\O, d^\alpha)}+ \lambda \|h\|_{W^{1,1}_\o (\O, d^\alpha, d^{\alpha+1})} : f= g+h , g\in L^1_\o(\O, d^\alpha), h \in W^{1,1}_\o(\O, d^\alpha, d^{\alpha+1})\}.
\end{equation}

As announced, we will obtain the characterization 
$$
(L^1_\o(\O, d^{\alpha}), W^{1,1}_\o(\Omega, d^{\alpha }, d^{\alpha+1}))_{s,1} =\widetilde W^{s,1}_\o(\O, d^{\alpha}, d^{\alpha+s})
$$
with equivalence of norms. The norm of the latter space is also equivalent to that of  
$$W^{s,1}_\o (\O, d^\alpha, \delta^{\alpha+s})= \{ f\in L^1_\o(\O, d^\alpha) : |f|_{ W^{s,1}_{\o}(\O, \delta^{\alpha+s})} <\infty \}$$
where $\delta(x,y)=\min\{d(x), d(y)\}$ and
$$
 |f|_{ W^{s,1}_{\o}(\O, \delta^{\alpha+s})} = \int_\O \int_\O \frac{|f(x)-f(y)|}{|x-y|^{n+s}} \, \delta(x,y)^{\alpha+s} \, dy \, \o(x) \, dx.
$$

The proof of this result is contained in the following lemma. Observe that it implies, in particular, that the norms of the spaces $\widetilde W^{s,1}_\o(\O, d^{\alpha}, d^{\alpha+s})$ for different values of $0<\tau<1$ are all equivalent. 

\begin{lemma}\label{equiv-normas}
Let $\O$ be a  domain, $\O\neq\R^n$, $0<s<1$, and $\alpha \ge 0$. Then,
$$
\widetilde W^{s,1}_\o(\O, d^{\alpha}, d^{\alpha+s}) = W^{s,1}_\o (\O, d^\alpha, \delta^{\alpha+s})
$$
with equivalent norms.  
\end{lemma}
\begin{proof}
Fix   $0<\tau < 1$. Observe that, whenever $|x-y|<\tau d(x)$,  one has  $d(x)\sim d(y)$ and, therefore, 

$$
\int_\Omega \int_{|x-y|<\tau d(x)} \frac{|f(x)-f(y)|}{|x-y|^{n+s}} \, dy \, d(x)^{\alpha+s} \omega(x) \, dx \lesssim  \int_\Omega \int_\Omega \frac{|f(x)-f(y)|}{|x-y|^{n+s}} \delta(x,y)^{\alpha+s} \, dy \,  \omega(x) \, dx.
$$

For the other inequality,  we have
\begin{align*} 
\int_\Omega \int_{|x-y|\ge \tau d(x)} \frac{|f(x)-f(y)|}{|x-y|^{n+s}} \delta(x,y)^{\alpha+s} \, dy \,  \omega(x) \, dx &\lesssim \int_\Omega \int_{|x-y|\ge \tau d(x)} \frac{|f(x)|+|f(y)|}{|x-y|^{n+s}}  \delta(x,y)^{\alpha+s}  \, dy \, \omega(x) \, dx.
\end{align*}

Now,
\begin{align*} 
 \int_\Omega \int_{|x-y|\ge \tau d(x)} \frac{|f(x)|}{|x-y|^{n+s}} \delta(x,y)^{\alpha+s}  \, dy \, \omega(x) \, dx &\lesssim  \int_\Omega \left(\int_{|x-y|\ge\tau d(x)} \frac{1}{|x-y|^{n+s}} \, dy \right) \, |f(x)| d(x)^{\alpha+s} \omega(x) \, dx\\
& \lesssim \|f\|_{L^1_{\o}(\Omega, d^\alpha)}.
 \end{align*}

And, since $|x-y|\ge \tau d(x) \Rightarrow d(y)\le |x-y|+d(x)\le (1+\frac{1}{\tau})  |x-y|$, by Fubini and \cite[Lemma (b)]{He}
\begin{align*} 
 \int_\Omega \int_{|x-y|\ge \tau d(x)} \frac{|f(y)|}{|x-y|^{n+s}} \delta(x,y)^{\alpha+s}  \, dy \, \omega(x) \, dx &\lesssim  \int_\Omega \left(\int_{|x-y|\ge\frac{\tau}{1+\tau}d(y)} \frac{\omega(x)}{|x-y|^{n+s}} \, dx \right) \, |f(y)| d(y)^{\alpha+s} \, dy\\
 &\lesssim  \int_\Omega M\omega(y) \, |f(y)| \,  d(y)^\alpha \, dy\\
& \lesssim \|f\|_{L^1_{\o}(\Omega, d^\alpha)}.
 \end{align*}
\end{proof}

Although our interpolation result holds for arbitrary  domains, we will then apply it to domains where a certain weighted  Poincar\'e inequality holds (see Theorem \ref{poincare}). A full characterization of domains supporting such  inequalities is still missing, but they are known to hold in bounded John domains. Moreover,  under  the additional assumption of a \emph{separation property}, this is exactly  the larger class where they hold (see \cite[Theorem 2.1]{JKK}), so we recall their definition below.
\begin{defi}
A bounded domain $\O\subset\R^n$ is a John domain if for a
fixed $x_0\in\O$ and any $y\in\O$ there exists a rectifiable
curve given by $\g(\cdot,y)\,:\, [0,1]\to\O$
such that $\g(0,y)=y$ and $\g(1,y)=x_0$, and there exist constants
$\delta$ and $K$, depending only on the domain $\O$ and on $x_0$, such
that
$d(\g(s,y))\ge\delta s$ and
$ |\frac{\partial\g}{\partial s}(s,y)|\le K$.
\end{defi}

\section{Proof of our main theorem}

\begin{theorem}
Let $\O\subset \R^n$ be a domain, $\O\neq \R^n$, $0<s<1$, $\alpha\ge 0$, and $\o \in A_1$. Then 
$$(L^1_\omega(\Omega, d^\alpha), W^{1,1}_{\omega}(\Omega, d^\alpha, d^{\alpha+1}))_{s,1} = \widetilde W^{s,1}_\o(\O, d^{\alpha}, d^{\alpha+s}) = W^{s,1}_\o (\O, d^\alpha, \delta^{\alpha+s}) $$
with equivalence of norms.
\end{theorem}

\begin{proof} 
The proof follows by Lemma \ref{equiv-normas} and the following two lemmas. 
\end{proof}

\begin{lemma}
Let $\O\subset \R^n$ be a domain, $\O\neq \R^n$, $0<s<1$, $\alpha\ge 0$, and $\o \in A_1$. Then 
$$(L^1_\omega(\Omega), W^{1,1}_{\omega}(\Omega, d^\alpha, d^{\alpha+1}))_{s,1} \subseteq \widetilde W^{s,1}_{\omega}(\Omega, d^{\alpha+s}).$$
\end{lemma}

\begin{proof}

By Lemma \ref{equiv-normas} we may take   $\tau=\frac1{16}$ and, rewriting the seminorm in a similar fashion as in \cite[Theorem 5.2]{GKS}, we have
\begin{align}
\int_\Omega &\int_{|x-y|<\frac{d(x)}{16}} \frac{|f(y)-f(x)|}{|x-y|^{n+s}} \, dy \, d(x)^{\alpha+s} \omega(x) \, dx \nonumber \\
&= \int_\Omega \sum_{i=4}^\infty \int_{B(x, \frac{d(x)}{2^i}) \setminus B(x, \frac{d(x)}{2^{i+1}})} \frac{|f(y)-f(x)|}{|x-y|^{n+s}} \, dy \, d(x)^{\alpha+s} \omega(x) \, dx \nonumber \\
&\lesssim  \int_\Omega \sum_{i=4}^\infty \Big(\frac{d(x)}{2^i}\Big)^{-(n+s)} \int_{B(x, \frac{d(x)}{2^i}) \setminus B(x, \frac{d(x)}{2^{i+1}})} |f(y)-f(x)| \, dy \, d(x)^{\alpha+s} \omega(x) \, dx \nonumber \\
&\lesssim  \int_\Omega \sum_{i=4}^\infty  \Big(\frac{d(x)}{2^i}\Big)^{-s} \frac{1}{|B(x, \frac{d(x)}{2^i})|} \int_{B(x, \frac{d(x)}{2^i}) } |f(y)-f(x)| \, dy \, d(x)^{\alpha+s} \omega(x) \, dx \nonumber \\
&\lesssim  \int_\Omega \sum_{i=4}^\infty  \, 2^{is} \avgint_{B(x, \frac{d(x)}{2^i}) } |f(y)-f(x)| \, dy \, d(x)^\alpha \omega(x) \, dx \label{eq1}
\end{align}

Observe that 
\begin{align*}
\int_{2^{-i}}^{2^{-i+1}} &\avgint_{B(x, \lambda d(x))} |f(x)-f(y)| \, dy \frac{d\lambda}{\lambda^{1+s}}\\
&\gtrsim \int_{2^{-i}}^{2^{-i+1}} \frac{1}{|B(x, 2^{-i+1} d(x))|}\int_{B(x, 2^{-i} d(x))} |f(x)-f(y)| \, dy \frac{d\lambda}{\lambda^{1+s}}\\
&\gtrsim \int_{2^{-i}}^{2^{-i+1}} \frac{1}{2^{(-i+1)(1+s)}} \frac{1}{|B(x, 2^{-i} d(x))|}\int_{B(x, 2^{-i} d(x))} |f(x)-f(y)| \, dy \, d\lambda\\
&\gtrsim \frac{2^{-i}}{2^{(-i+1)(1+s)}} \, \avgint_{B(x, 2^{-i} d(x))} |f(x)-f(y)| \, dy \\
&\gtrsim 2^{is} \, \avgint_{B(x, 2^{-i} d(x))} |f(x)-f(y)| \, dy 
\end{align*}

So that, plugging this into \eqref{eq1}, we obtain
\begin{align}
\int_\Omega &\int_{|x-y|<\frac{d(x)}{16}} \frac{|f(y)-f(x)|}{|x-y|^{n+s}} \, dy \, d(x)^{\alpha+s} \omega(x) \, dx \nonumber \\
&\lesssim  \int_\Omega \sum_{i=4}^\infty  \, 2^{is} \avgint_{B(x, \frac{d(x)}{2^i}) } |f(y)-f(x)| \, dy \, d(x)^\alpha \omega(x) \, dx \nonumber\\
&\lesssim  \int_\Omega \sum_{i=4}^\infty  \,  \int_{2^{-i}}^{2^{-i+1}} \avgint_{B(x,\lambda d(x))} |f(x)-f(y)| \, dy \, \frac{d\lambda}{\lambda^{1+s}} \, d(x)^\alpha \omega(x) \, dx  \nonumber\\
&\lesssim  \int_\Omega  \int_0^\frac18 \avgint_{B(x,\lambda d(x))} |f(x)-f(y)| \, dy \, \frac{d\lambda }{\lambda^{1+s}} \,  d(x)^\alpha \omega(x) \, dx  \nonumber\\
&= \int_0^\frac18 E(f,\lambda) \frac{d\lambda}{\lambda^{1+s}} \label{cotaE}
\end{align}
with $$ E(f,\lambda) :=\int_\Omega  \avgint_{B(x,\lambda d(x))} |f(x)-f(y)| \, dy  \, d(x)^\alpha \omega(x) \, dx.$$

Now, for each $\lambda \in (0,\frac18)$, pick a decomposition $f=g_\lambda+h_\lambda$, with $g_\lambda \in L^1_\omega(\O, d^\alpha), h_\lambda \in W^{1,1}_{\omega}(\O,d^\alpha,d^{\alpha+1})$, and  $\| d^\alpha g_\lambda\|_{L^1_\omega(\O)} + \lambda \|d^{\alpha+1}\nabla h_\lambda\|_{L^1_{\omega}(\O)}\le 2K(f,\lambda)$, with $K(f,\lambda)$ as in \eqref{K2}. We remark that we may assume that $f\not\equiv 0$ to guarantee that $K(f,\lambda)>0$. Since $E(f,\lambda)\le E(g_\lambda, \lambda)+ E(h_\lambda ,\lambda)$, we may bound these terms separately. 

Observe that, for $x\in \O$, $y\in B(x, \lambda d(x))$ and $\lambda \in (0, \frac18)$,  
$$d(x) \le d(y) + |x-y| < d(y) + \lambda d(x)  \Rightarrow   d(x) < \frac87 d(y)$$
 $$d(y) \le d(x) + |x-y| < d(x) +\lambda d(x) < \frac98 d(x)$$

Therefore, we have that $x \in B(y, \frac87 \lambda d(y))$ and $d(x) \sim d(y)$. Hence, by Fubini,

\begin{align} \label{cotayx}
E(g_\lambda, \lambda) &\lesssim \int_\Omega \frac{1}{(\lambda d(x))^n} \int_{B(x,\lambda d(x))} (|g_\lambda(x)| + |g_\lambda(y)|) \, dy \, d(x)^\alpha \omega(x) \, dx \\
&\lesssim \int_\Omega |g_\lambda(x)| \, d(x)^\alpha \omega(x) \, dx + \int_\Omega  \frac{1}{(\lambda d(x))^n}  \int_{B(x, \lambda d(x))}  |g_\lambda(y)| \, dy \, d(x)^\alpha \omega(x) \, dx\nonumber\\
&\lesssim \int_\Omega |g_\lambda(x)| \, d(x)^\alpha \omega(x) \, dx + \int_\Omega |g_\lambda(y)| \frac{d(y)^\alpha}{(\lambda d(y))^n} \int_{B(y,\frac87 \lambda d(y))} \omega(x) \, dx \, dy\nonumber\\
&\lesssim \int_\Omega |g_\lambda (x)| \, d(x)^\alpha \omega(x) \, dx + \int_\Omega |g_\lambda(y)| \, d(y)^\alpha M\omega(y)  \, dy\nonumber\\
&\lesssim   \| g_\lambda\|_{L^1_\omega(\O, d^\alpha)},\nonumber
\end{align}
where in the last inequality we have used that $M\omega(y) \lesssim \omega(y)$ almost everywhere, because $\omega \in A_1$.

To bound $E(h_\lambda ,\lambda)$, let  $B=B(x, \lambda d(x))$ and $h_{\lambda,B}=\frac{1}{|B|}\int_B h_\lambda(z) \, dz$.  
Then, for any $y\in B$, by \cite[Lemma 7.16]{GT} we may write 
\begin{align*}
|h_\lambda (x)-h_\lambda(y)|&\le |h_\lambda(x)-h_{\lambda,B}| +|h_{\lambda,B}-h_\lambda(y)|\lesssim  \int_{B} \frac{|\nabla h_\lambda(z)|}{|x-z|^{n-1}} \, dz +   \int_{B} \frac{|\nabla h_\lambda(z)|}{|y-z|^{n-1}} \, dz.
\end{align*} So, we obtain 
\begin{align}
E(h_\lambda, \lambda) & =\int_\Omega   \frac{1}{(\lambda d(x))^n}  \int_{B(x,\lambda d(x))} |h_\lambda (x)-h_\lambda(y)| \, dy  \, d(x)^\alpha \omega(x) \, dx\\
&\lesssim \int_\Omega \frac{1}{(\lambda d(x))^n} \int_{B(x,\lambda d(x))} \Big( \int_{B(x,\lambda d(x))} \frac{|\nabla h_\lambda(z)|}{|x-z|^{n-1}} \, dz +  \int_{B(x,\lambda d(x))} \frac{|\nabla h_\lambda(z)|}{|y-z|^{n-1}} \, dz \Big) \, dy \, d(x)^\alpha \omega(x) \, dx\nonumber\\
&= I + II \label{Eh}
\end{align}

If  $z\in B(x,\lambda d(x))$ and $\lambda\in(0,\frac18)$, observe that by the computations right before \eqref{cotayx} (replacing $y$ by $z$), we can deduce that $x\in B(z, \frac87 \lambda d(z))$ and that $d(x)\sim d(z)$. Hence, by Fubini and \cite[Lemma (a)]{He},

\begin{align*}
I &\lesssim \int_\Omega  \int_{B(x, \lambda d(x))} \frac{|\nabla h_\lambda(z)|}{|x-z|^{n-1}} \, dz \, d(x)^\alpha \omega(x) \, dx\\
&\lesssim \int_\Omega \int_{B(z, \frac87 \lambda d(z))} \frac{\omega(x)}{|x-z|^{n-1}} \, dx \, |\nabla h_\lambda(z)| \,  d(z)^\alpha dz\\
&\lesssim \int_\Omega \lambda d(z) M\omega(z) |\nabla h_\lambda(z)| \, d(z)^\alpha dz\\
&\lesssim   \lambda \| \nabla h_\lambda\|_{L^1_ {\omega }(\O,d^{\alpha+1} )}.
\end{align*}

Similarly, to bound $II$, recall that for $y\in B(x, \lambda d(x))$ and $\lambda \in (0, \frac18)$, we have that $x\in B(y,\frac87 \lambda d(y))$ and that $\frac78 d(x) <d(y) <\frac98 d(x)$, so that, if $z\in B(x,\lambda d(x))$,
\begin{equation*}
|z-y|<|z-x|+|x-y|< \lambda d(x) + \frac87 \lambda d(y) < \frac{16}7 \lambda d(y).
\end{equation*}
Therefore, by Fubini, 
\begin{align*}
II &=  \int_\Omega \frac{1}{(\lambda d(x))^n} \int_{B(x,\lambda d(x))}   \int_{B(x,\lambda d(x))}\frac{|\nabla h_\lambda (z)|}{|y-z|^{n-1}} \, dz \, dy \, d(x)^\alpha \omega(x) \, dx\\
&\lesssim  \int_\Omega \frac{1}{(\lambda d(y))^n} \int_{B(y, \frac{16}7 \lambda d(y))}   \int_{B(y, \frac87 \lambda d(y))}  \omega(x) \, dx \, \frac{|\nabla h_\lambda(z)|}{|y-z|^{n-1}} \, dz \, d(y)^\alpha dy \\
&\lesssim  \int_\Omega  \int_{B(y, \frac{16}7 \lambda d(y))}  M\omega(y)  \, \frac{|\nabla h_\lambda (z)|}{|y-z|^{n-1}} \, dz \, d(y)^\alpha dy\\
&\lesssim  \int_\Omega  \int_{B(y, \frac{16}7 \lambda d(y))}  \omega(y)  \, \frac{|\nabla h_\lambda (z)|}{|y-z|^{n-1}} \, dz \, d(y)^\alpha dy.
\end{align*}

Now, observe that for $z \in B(y, \frac{16}7 \lambda d(y))$  and $\lambda \in (0, \frac18)$ we have 
$$d(y)<d(z)+|z-y| <d(z)+\frac{16}7 \lambda d(y)<d(z)+\frac27 d(y) \Rightarrow d(y) <\frac75 d(z)$$
$$d(z)<d(y)+|z-y|<d(y)+ \frac{16}{7}\lambda d(y)<\frac{9}{7} d(y)$$ 
Hence,  $ d(y) \sim d(z)$ and  $|z-y|<\frac{16}7 \lambda d(y) < \frac{16}5 \lambda d(z)$, so by Fubini and \cite[Lemma (a)]{He}, 

\begin{align*}
II &\lesssim  \int_\Omega  \int_{B(z, \frac{16}5 \lambda d(z))} \frac{ \omega(y)  }{|y-z|^{n-1}} \, dy \, |\nabla h_\lambda(z)| \, d(z)^\alpha dz\\
&\lesssim  \int_\Omega  \lambda d(z)  M\omega (z) \, |\nabla h_\lambda(z)| \, d(z)^\alpha dz\\
&\lesssim   \lambda \| \nabla h_\lambda\|_{L^1_{\omega }(\O, d^{\alpha+1})}.
\end{align*}

Finally, we arrive at 
\begin{align*}
\int_\Omega \int_{|x-y|<\frac{d(x)}{16}} \frac{|f(y)-f(x)|^p}{|x-y|^{n+sp}} \, dy \, d(x)^{\alpha+s} \, \omega(x) dx &\lesssim   \int_0^1 \Big( \| g_\lambda\|_{L^1_\omega(\O, d^\alpha)} + \lambda \| \nabla h_\lambda\|_{L^1_\omega(\O, d^{\alpha+1} )} \Big) \frac{d\lambda}{\lambda^{1+s}} \\
&\lesssim   \int_0^1 \lambda^{-s} K(f,\lambda) \, \frac{d\lambda}{\lambda}.
\end{align*}

This completes the proof.
\end{proof}

\begin{lemma}
Let $\O\subset \R^n$ be a domain, $\O\neq \R^n$, $0<s<1$, $\alpha\ge 0$, and $\o \in A_1$. Then 
$$\widetilde W^{s,1}_{\omega}(\Omega, d^{\alpha+s}) \subseteq (L^1_\omega(\Omega), W^{1,1}_{\omega}(\Omega, d^\alpha, d^{\alpha+1}))_{s,1}.$$
\end{lemma}

\begin{proof}
Observe first that, by  the trivial bound $K(f,\lambda)\le \|f\|_{L^1_\o(\O, d^\alpha)}$, we always have
$$
\int_1^\infty \lambda^{-s} K(f, \lambda) \, \frac{d\lambda}{\lambda} \le \|f\|_{L^1_\o(\Omega, d^\alpha)} \int_1^\infty \lambda^{-s} \, \frac{d\lambda}{\lambda} \lesssim \|f\|_{L^1_\o(\Omega, d^\alpha)}.
$$

Also,  for a given decomposition $f=g+h$ as in \eqref{K2}, 
\begin{align*}
\int_0^1 \lambda^{1-s} \|h\|_{L^1_\o(\Omega, d^\alpha)} \, \frac{d\lambda}{\lambda} & \lesssim  \int_0^1 \lambda^{1-s} \|f\|_{L^1_\o(\Omega, d^\alpha)}  \frac{d\lambda}{\lambda} + \int_0^1 \lambda^{1-s} \|g\|_{L^1_\o(\Omega, d^\alpha)} \frac{d\lambda}{\lambda}\\
&\lesssim    \|f\|_{L^1_\o(\Omega, d^\alpha)} +  \int_0^1 \lambda^{-s} \|g\|_{L^1_\o(\Omega, d^\alpha)} \frac{d\lambda}{\lambda}.
\end{align*}

Therefore,
\begin{equation}
\label{simplificada} \int_0^\infty \lambda^{-s} K(\lambda, f) \, \frac{d\lambda}{\lambda}  \lesssim  \|f\|_{L^1_\o(\Omega, d^\alpha)} +\int_0^1 \lambda^{-s} (\|g\|_{L^1_\o(\Omega, d^\alpha)} + \lambda \|\nabla h\|_{L^1_\o(\Omega, d^{\alpha+1})}) \frac{d\lambda}{\lambda},
\end{equation}
and to prove the claimed embedding it suffices to bound the integral on the right-hand side for specific choices of $g$ and $h$ that we will define below.

As in  \cite[Section 4]{ADD},  given a cube $Q\subset \R^n$, the distance from $Q$ to the boundary of $\Omega$ is denoted by $d(Q,\partial\O)$, while $\mbox{diam}(Q)$ and $\ell_Q$ are the diameter and length of the edges of  $Q$, respectively. 
We pick a Whitney decomposition $\mathcal{W}=\{Q\}$ of $\O$ and,  for every fixed $0<\lambda\le 1$, we build a new dyadic decomposition  $\mathcal{W}^\lambda=\{Q^\lambda\}$   by dividing each $Q \in \mathcal{W}$  in such a way that $\frac12 \lambda \ell_Q\le \ell_{Q^\lambda}\le  \lambda \ell_Q$. Notice that, in particular, this means that $\frac12 \lambda \mbox{diam} (Q)\le  \mbox{diam}({Q^\lambda}) \le  \lambda \mbox{diam} (Q)$.  The center of  $Q^\lambda_j$ in this new partition  is denoted by $x_j^\lambda$,  and we  write $\ell^\lambda_j$ instead of $\ell_{Q^\lambda_j}$. 

For each $\mathcal{W}^\lambda$ we can define the covering of expanded cubes  ${\mathcal{W}^\lambda}^*=\{(Q_j^\lambda)^*\}$ where   $Q^*$  is the cube with the same center as $Q$ but expanded by a factor $9/8$. Observe that it satisfies  $\sum_j \chi_{(Q_j^\lambda)^*}(x) \le C$ for every $x\in \Omega$, and that, for $x \in (Q_j^\lambda)^*$,
\begin{equation}
\label{cotasd}
\frac34 \frac{\mbox{diam}(Q_j^\lambda)}{\lambda} \le d(x) \le \frac{41}4 \frac{\mbox{diam}(Q_j^\lambda)}{\lambda}. 
\end{equation}
 
Associated to this covering we consider a smooth partition of unity  $\{\psi_j^\lambda\}$ such that  $\mbox{supp} (\psi_j^\lambda)\subset (Q_j^\lambda)^*$, $0\le \psi_j^\lambda\le 1$,     $\sum_j \psi_j^\lambda = 1$ in $\Omega$, and $\|\nabla \psi_j^\lambda\|_\infty \le \frac{C}{\ell_j^\lambda}$.

For a given (fixed)  $C^\infty$ function  $\varphi \ge 0$ such that $\mbox{supp}(\varphi) \subset B(0,\frac14)$ and $\int \varphi =1$, and for each $t>0$, we define $\varphi_t(x)=t^{-n} \varphi(t^{-1}x)$. Then, for a given $f \in \widetilde W^{s,1}_\o(\Omega, d^\alpha, d^{\alpha+s})$  we  define
\begin{equation}
\label{ht}
h^\lambda(y)=\sum_j f_j^\lambda \psi_j^\lambda(y),
\end{equation}
 with  
 $$f_j^\lambda = \int_{\mathbb{R}^n} f* \varphi_{\ell^\lambda_j}(z) \varphi_{\ell^\lambda_j} (z-x_j^\lambda) \, dz,$$
 which is a smooth approximation of $f$. Moreover, by \cite[page 9]{ADD}, one has that, for  $y \in (Q_j^\lambda)^*$,
\begin{align*}
|f(y)-f_j^\lambda| &\le \int_0^1 \int_{|x-y|<Ct\ell_j^\lambda} \int_{|x-w|<\frac{d(x)}{2}} |f(x)-f(w)| \chi_{|x-w|<\frac14 t \ell_j^\lambda} \, dw \, \frac{\chi_{(Q_j^\lambda)^*}(x)}{t^{2n+1}} \, dx \, dt \, (\ell_j^\lambda)^{-2n}.
\end{align*}

Since the family ${\mathcal{W}^\lambda}^*$ has finite overlapping, we have that
$$
\|f-h^\lambda\|_{L^1_\o(\Omega, d^\alpha )} \le C \sum_j \| f-f_j^\lambda\|_{L^1_\o((Q_j^\lambda)^*, d^\alpha )}.
$$

Using that, for $x\in (Q_j^\lambda)^*$,  $\ell_j^\lambda \sim \lambda d(x)$ and that $|x-y|<Ct\ell_j^\lambda \Rightarrow d(y)\lesssim d(x)$, we have

\begin{align*}
\int_0^1 & \lambda^{-s} \| f-h^\lambda\|_{L^1_\omega(\Omega, d^\alpha)} \frac{d\lambda}{\lambda} \\
& \lesssim \int_0^1\sum_j \int_{(Q_j^\lambda)^*} \int_0^1 \int_{|x-y|<Ct\ell_j^\lambda} \int_{|x-w|<\frac{d(x)}{2}} |f(x)-f(w)| \,  \chi_{|x-w|<\frac14 t \ell_j^\lambda} \, dw \, \frac{ (\ell_j^\lambda)^{-2n} \lambda^{-s}}{t^{2n+1}} \, dt \, dx \, d(y)^\alpha \omega(y) \, dy \, \frac{d\lambda}{\lambda}\\
& \lesssim \int_0^1\sum_j \int_{(Q_j^\lambda)^*} \int_0^1  \int_{|x-w|<\frac{d(x)}{2}} |f(x)-f(w)| \,  \chi_{|x-w|<\frac14 t \ell_j^\lambda} \, dw \Big( \int_{|x-y|<Ct\ell_j^\lambda} \omega(y) \, dy\Big) \frac{ (\ell_j^\lambda)^{-2n} \lambda^{-s}}{t^{2n+1}} \, dt \, d(x)^\alpha dx \, \frac{d\lambda}{\lambda}\\
& \lesssim \int_0^1\sum_j \int_{(Q_j^\lambda)^*} \int_0^1  \int_{|x-w|<\frac{d(x)}{2}} |f(x)-f(w)| \,  \chi_{|x-w|<\frac14 t \ell_j^\lambda} \, dw (t \ell_j^\lambda)^n M\omega(x) \frac{ (\ell_j^\lambda)^{-2n} \lambda^{-s}}{t^{2n+1}} \, dt \, d(x)^\alpha dx \, \frac{d\lambda}{\lambda}\\
& \lesssim   \int_0^1\sum_j \int_{(Q_j^\lambda)^*} \int_0^1  \int_{|x-w|<\frac{d(x)}{2}} |f(x)-f(w)| \,  \chi_{|x-w|<\frac14 t \ell_j^\lambda} \, dw  \, \omega(x) \frac{ (\ell_j^\lambda)^{-n} \lambda^{-s}}{t^{n+1}} \, dt \, d(x)^\alpha dx \, \frac{d\lambda}{\lambda}\\
&  \lesssim   \int_0^1\sum_j \int_{(Q_j^\lambda)^*} \int_0^1  \int_{|x-w|<\frac{d(x)}{2}} |f(x)-f(w)| \,  \chi_{|x-w|<\frac{c}{4} t \lambda d(x)} \, dw  \, \omega(x) \frac{ d(x)^{\alpha-n} \lambda^{-n-s}}{t^{n+1}} \, dt \, dx \, \frac{d\lambda}{\lambda}\\
& \lesssim   \int_0^1 \int_\Omega \int_0^1  \int_{|x-w|<\frac{d(x)}{2}} |f(x)-f(w)| \,  \chi_{|x-w|<\frac{c}{4} t \lambda d(x)} \, dw  \, \omega(x) \frac{ d(x)^{\alpha-n} \lambda^{-n-s}}{t^{n+1}} \, dt \, dx \, \frac{d\lambda}{\lambda}\\
& \lesssim    \int_\Omega \int_0^1  \int_{|x-w|<\frac{d(x)}{2}} |f(x)-f(w)| \, dw \int_{\frac{4 |x-w|}{c t d(x)}}^\infty \lambda^{-n-s-1} \, d\lambda \,  \frac{ d(x)^{\alpha-n}}{t^{n+1}} \, dt  \, \omega(x) \, dx \\
& \lesssim   \frac{1 }{n+s}  \int_\Omega \int_0^1  \int_{|x-w|<\frac{d(x)}{2}} \frac{|f(x)-f(w)|}{|x-w|^{n+s}} \, dw \, t^{s-1} d(x)^{\alpha+s}    \, dt  \, \omega(x) \, dx \\
& \lesssim  \frac{ 1}{s(n+s)} \int_\Omega \int_{|x-w|<\frac{d(x)}{2}} \frac{|f(x)-f(w)|}{|x-w|^{n+s}} \, dw \,  d(x)^{\alpha+s}  \omega(x) \, dx 
\end{align*}

On the other hand, recalling that  $\mbox{supp} (\psi_j^\lambda)\subset (Q_j^\lambda)^*$, that $\|\nabla \psi_j^\lambda\|_\infty \le \frac{C}{\ell_j^\lambda}$,
and that $\nabla (\sum_j \psi_j^\lambda) =0$,
$$
|\nabla h^\lambda(y)| = \Big| \sum_j f_j^\lambda \nabla \psi_j^\lambda (y) \Big| \lesssim  \sum_j |f_j^\lambda -f(y)|  \frac{1}{\ell_j^\lambda} \chi_{(Q_j^\lambda)^*}(y).
$$
Therefore, 
\begin{align*}
\int_0^1 & \lambda^{1-s} \|\nabla h^\lambda\|_{L^1_{\omega}(\Omega, d^{\alpha+1 })} \frac{d\lambda}{\lambda} \\
& \lesssim \int_0^1 \sum_j \lambda^{1-s} \|(f-f_j^\lambda)(\ell_j^\lambda)^{-1}\|_{L^1_\o((Q_j^\lambda)^*, d^{\alpha+1} )} \frac{d\lambda}{\lambda}\\
& \lesssim \int_0^1 \sum_j \lambda^{1-s} \| (f-f_j^\lambda) \lambda^{-1}\|_{L^1_\omega((Q_j^\lambda)^*, d^\alpha)} \frac{d\lambda}{\lambda}\\
& \lesssim \int_0^1 \sum_j \lambda^{-s} \|(f-f_j^\lambda)\|_{L^1_\omega((Q_j^\lambda)^*, d^\alpha)} \frac{d\lambda}{\lambda}
\end{align*}
so this term can be bounded as before.

Summing up, 
\begin{equation}\label{cota-int}
\int_0^1 \lambda^{-s} \Big( \| f-h^\lambda\|_{L^1_\omega(\O, d^\alpha)} + \lambda \|\nabla h^\lambda \|_{L^1_{\omega }(\O, d^{\alpha+1})} \Big) \frac{d\lambda}{\lambda}  \lesssim \frac{1 }{s(n+s)} \int_\Omega \int_{|x-y|<\frac{d(x)}{2}} \frac{|f(y)-f(x)|}{|x-y|^{n+s}} \, dy \, d(x)^{\alpha+s} \o(x) dx.
\end{equation}

This concludes the proof.
 
\end{proof}

\begin{remark}\label{remark-tau}
It is immediate that inequality \eqref{cota-int} also holds  for every $\frac12<\tau<1$. If one wishes to obtain it for $0<\tau<\frac12$, it suffices to choose $supp(\varphi) \subset B(0, \varepsilon)$ for sufficiently small $\varepsilon$ in the above proof, as the reader can check by following the computations in \cite[page 8]{ADD}.
\end{remark}

\section{Applications to fractional Poincar\'e inequalities}

In the forthcoming results we will make use of two well-known properties of weighted norms contained in the following lemma. We include a proof for the sake of completeness.
\begin{lemma}
Let $\O$ be a bounded domain, $\nu$ a locally integrable nonnegative function,  and  $f_\nu = \frac{1}{\nu(\Omega)}\int_\Omega f(x) \nu(x) \, dx$. Then,
\begin{enumerate}
\item $\inf_{c\in \mathbb{R}} \|f-c\|_{L_\nu^1(\Omega)}  \sim \|f-f_\nu\|_{L^1_\nu(\O)}$.
\item $\|f - f_\nu\|_{L^1_\nu(\O)}\le 2 \|f\|_{L^1_\nu(\O)}$.
\end{enumerate}
\end{lemma}
\begin{proof}
1) It is immediate that  $\inf_{c\in \mathbb{R}} \|f-c\|_{L_\nu^1(\Omega)}  \le \|f-f_\nu\|_{L^1_\nu(\O)}$. For the other inequality, it suffices to observe that, for any $c\in \R$,

\begin{align*}
\|f-f_\nu\|_{L^1_\nu(\O)} & \le \|f-c\|_{L^1_\nu (\O)} + \|c-f_\nu\|_{L^1_\nu (\O)}\\
& \le \|f-c\|_{L^1_\nu (\O)} + \nu(\O) \left| c - \frac{1}{\nu(\O)} \int_\O f(x) \nu(x) \, dx\right| \\
&\le 2 \|f-c\|_{L^1_\nu(\O)}.
\end{align*}

2) Write
\begin{align*}
\|f-f_\nu\|_{L^1_\nu(\O)} &\le \frac{1}{\nu(\Omega)} \int_\Omega \int_\Omega |f(x)-f(y)| \, \nu(x) \nu(y) \, dx \, dy\\
&\le \int_\Omega |f(x)| \nu(x) \, dx + \int_\Omega |f(y)| \nu(y) \, dy = 2 \|f\|_{L^1_\nu(\O)}.
\end{align*}
\end{proof}

\begin{theorem}\label{poincare}
Let  $\Omega$ be a bounded domain, $\alpha \ge 0$,   $\omega\in A_1$ and   $\| \nabla f\|_{L^1_\o(\O, d^{\alpha+1})} <\infty$. Then, the following are equivalent:
\begin{enumerate}
\item $\displaystyle{\inf_{c\in \mathbb{R}} \| f-c \|_{L^1_\o(\Omega, d^\alpha)} \lesssim   \| \nabla f\|_{L^1_\o(\Omega, d^{\alpha+1})}}$
\item  $\displaystyle{\inf_{c\in \mathbb{R}} \| f-c\|_{L^1_\o(\Omega, d^\alpha)} \lesssim   \frac{(1-s)}{s(n+s)} \int_\Omega \int_{|x-y|<\tau d(x)} \frac{|f(x)-f(y)|}{|x-y|^{n+s}} \, d(x)^{\alpha+s} \omega(x) \, dy \, dx}$  for every $0<\tau<1$ and every   ${0<s<1}$
\item  $\displaystyle{\inf_{c\in \mathbb{R}} \|f-c\|_{L^1_\o(\Omega, d^\alpha)} \lesssim   \frac{(1-s)}{s(n+s)}\int_\Omega \int_{|x-y|<\tau d(x)} \frac{|f(x)-f(y)|}{|x-y|^{n+s }} \, d(x)^{\alpha+s} \omega(x) \, dy \, dx}$  for every $0<\tau<1$ and  some ${0<s<1}$
\end{enumerate}
\end{theorem}
\begin{proof}

 $1) \Rightarrow 2)$ This  is a straightforward generalization of an unpublished result by Oscar Dom\'inguez Bonilla for the case $\alpha=0, \o\equiv 1$. Define  $h^\lambda$ as in \eqref{ht} and $\nu= d^\alpha \o$. By hypothesis and the previous lemma,
 \begin{align*}
 \|f- f_\nu\|_{L^1_\nu(\O)} &\lesssim \|f-h^\lambda - ( f_\nu - (h^\lambda)_\nu)\|_{L^1_\nu(\O)}+ \|h^\lambda -  (h^\lambda)_\nu\|_{L^1_\nu(\O)}\\
 &\lesssim \|f-h^\lambda\|_{L^1_\nu(\O)} + \|d \nabla h^\lambda\|_{L^1_\nu(\O)}.
 \end{align*}

Then, for $\lambda\le 1$, 
\begin{align*}
   \lambda\|f- f_\nu\|_{L^1_\nu(\Omega)} &\lesssim \lambda \|f-h^\lambda\|_{L^1_\nu(\Omega)} + \lambda \|d \nabla h^\lambda\|_{L^1_\nu(\Omega)}\\
&\lesssim \| f-h^\lambda\|_{L^1_\o(\Omega, d^\alpha)} + \lambda \|\nabla h^\lambda\|_{L^1_\o(\Omega, d^{\alpha+1} )}.
\end{align*}
Therefore,
$$
\int_0^1 \lambda^{-s+1}  \|f- f_\nu\|_{L^1_\nu(\Omega)} \frac{d\lambda}{\lambda} \lesssim \int_0^1 \lambda^{-s} \Big(\|  f-h^\lambda\|_{L^1_\o(\Omega, d^\alpha)} + \lambda \|  \nabla h^\lambda\|_{L^1_\o(\Omega, d^{\alpha+1})}\Big) \frac{d\lambda}{\lambda}
$$
for every $0<s<1$. Then, by \eqref{cota-int} and Remark \ref{remark-tau},
$$
\frac{1}{(1-s)} \|f- f_\nu\|_{L^1_\nu(\Omega)} \lesssim \int_0^1 \lambda^{-s} \Big( \| f-h^\lambda\|_{L^1_\o(\Omega, d^\alpha)} + \lambda \| \nabla h^\lambda \|_{L^1_\o(\Omega, d^{\alpha+1})}\Big) \frac{d\lambda}{\lambda} \lesssim \frac{1 }{s(n+s)} |f|_{\widetilde W^{s,1}_\o(\O, d^{\alpha+s})}
$$
for  all values of $0<\tau<1$. 

So that, again by the previous lemma,
 $$
\inf_{c\in \R} \| f- c\|_{L^1_\o(\Omega, d^\alpha)}  \lesssim   \frac{(1-s)}{s(n+s)} |f|_{\widetilde W^{s,1}_\o(\O,  d^{\alpha+s})}.
$$

 $2) \Rightarrow 3)$  Trivial.
 
 $3) \Rightarrow 1)$ By \eqref{cotaE}, for $\tau=\frac1{16}$ and repeating for $f$ the computations previoulsy made for $h_\lambda$ in \eqref{Eh}, we get
 $$(1-s) \int_\Omega \int_{|x-y|<\tau d(x)} \frac{|f(y)-f(x)|}{|x-y|^{n+s}} \, dy \, d(x)^{\alpha+s} \omega(x) \, dx \lesssim    \| \nabla f\|_{L^1_{\omega}(\O, d^{\alpha+1})},$$
 and the result follows.
  
  \end{proof}

By the previous theorem one immediately has:

\begin{coro}
Let $\O\subset \R^n$ be a bounded John domain, $\alpha \ge 0$,  $\omega\in A_1$, and  $\|\nabla f\|_{L^1_\o(\O, d^{\alpha+1} )} <\infty$. Then, 
 $$\inf_{c\in \mathbb{R}} \| f-c\|_{L^1_\o(\Omega, d^\alpha)} \lesssim    \frac{(1-s)}{s(n+s)} \int_\Omega \int_{|x-y|<\tau d(x)} \frac{|f(x)-f(y)|}{|x-y|^{n+s}} \, d(x)^{\alpha+s } \omega(x) \, dy \, dx$$
 for every $0<\tau<1$ and every $0<s<1$. 
\end{coro}
\begin{proof}
It suffices to check that
 $$\inf_{c\in \mathbb{R}} \|f-c \|_{L^1_\o(\Omega, d^\alpha )} \lesssim   \| \nabla f\|_{L^1_\o(\Omega, d^{\alpha+1})}.$$
This can be seen with a slight modification of the proof in \cite[Theorem 3.4]{DD1} (which is the case $\alpha=0$), we briefly indicate the necessary steps.

Following that proof, by duality it suffices to bound $\int_\O (f-f_\varphi)(y) g(y) d(y)^\alpha \, dy$ for any $g$ such that $\| \o^{-1} g\|_{L^\infty(\O)}\le 1$. 

As in \cite[equation (3.2)]{DD1} and noting that $|x-y|\le Cd(x) \Rightarrow d(y)\lesssim d(x)$, we have
\begin{align}
\int_\O |(f(y)-f_\varphi) g(y)| d(y)^\alpha  \, dy &\lesssim \int_\O \int_{|x-y|\le C d(x)} \frac{|g(y)| \chi_\O(y)}{|x-y|^{n-1}} \, dy |\nabla f(x)|  d(x)^\alpha \, dx \nonumber \\
&\lesssim \int_\O M(\chi_\O g)(x) d^{\alpha+1}(x) |\nabla f(x)| \, dx \label{cotahedberg} \\
&\lesssim \|\o^{-1} M(\chi_\O g)  \|_{L^\infty(\O)} \|\o \, d^{\alpha+1} \nabla f \|_{L^1(\O)} \nonumber \\
&\lesssim   \|\o^{-1} g\|_{L^\infty(\O)} \|  \nabla f\|_{L^1_\o(\O, d^{\alpha+1})} \label{cotamaximal}\\
&\lesssim   \|  \nabla f\|_{L^1_\o(\O, d^{\alpha+1})} \nonumber
\end{align}
 where in \eqref{cotahedberg} we have used  \cite[Lemma (a)]{He}, and in \eqref{cotamaximal} we have used \cite[Theorem 4]{Mu} (actually, the remark at the end of  \cite[Section 7]{Mu} regarding its extension to the $n$-dimensional case).
\end{proof}


\begin{thebibliography}{}

\bibitem{ADD} Acosta, G.; Drelichman, I.; Dur\'an, R.G. \emph{Weighted fractional Sobolev spaces as interpolation spaces in bounded domains}. Math. Nachr.  296 (2023), no. 9,
4374--4385.


\bibitem{B} Brezis, H. \emph{How to recognize constant functions. A connection with Sobolev spaces}. Russian Math. Surveys 57 (2002), no. 4, 693--708 

\bibitem{BBM} Bourgain, J.;  Brezis, H.;  Mironescu, P. \emph{Another look at Sobolev spaces}. Optimal control and partial differential equations, IOS, Amsterdam, 2001, pp. 439--455. 

\bibitem{BBM2} Bourgain, J.; Brezis, H.; Mironescu, P. \emph{Limiting embedding theorems for $W^{s,p}$ when $s\uparrow 1$ and applications}. J. Anal. Math. 87 (2002), 77--101.

\bibitem{DLTYY}Dom\'inguez, O.; Li, Y.; Tikhonov, S.; Yang, D., Yuan, W. \emph{A unified approach to self-improving property via $K$-functionals}. Preprint (2023). https://doi.org/10.48550/arXiv.2309.02597 


\bibitem{DD-fracpoincare} Drelichman, I.; Dur\'an, R. G. \emph{Improved Poincar\'e inequalities in fractional Sobolev spaces}. Ann. Acad. Sci. Fenn. Math. 43 (2018), no. 2, 885--903.

\bibitem{DD1}Drelichman, I.; Dur\'an, R. G. \emph{Improved Poincar\'e inequalities with weights}. J. Math. Anal. Appl. 347 (2008), no. 1, 286--293. 

\bibitem{DD-BBM} Drelichman, I.; Dur\'an, R. G. \emph{The Bourgain-Br\'ezis-Mironescu formula in arbitrary bounded domains}.  Proc. Amer. Math. Soc. 150 (2022), no. 2, 701--708. 

\bibitem{DIV} Dyda, B.; Ihnatsyeva, L.; V\"ah\"akangas, A. V. \emph{On improved fractional Sobolev-Poincar\'e inequalities}. Ark. Mat. 54 (2016), no. 2, 437--454.

\bibitem{GT} Gilbarg, D.; Trudinger, N. S. Elliptic partial differential equations of second order. Second edition. Grundlehren der mathematischen Wissenschaften, 224. Springer-Verlag, Berlin, 1983. 

\bibitem{GKS} Gogatishvili, A.; Koskela, P.; Shanmugalingam, N. \emph{Interpolation properties of Besov spaces defined on metric spaces}. 
Math. Nachr. 283 (2010), no. 2, 215--231.

\bibitem{G} Guo, CY. \emph{Fractional Sobolev-Poincar\'e inequalities in irregular domains}. Chin. Ann. Math. Ser. B 38 (2017), 839--856 . 

\bibitem{He} Hedberg, L. I. \emph{On certain convolution inequalities}. Proc. Amer. Math. Soc. 36 (1972), 505--510.

\bibitem{HMPV} Hurri-Syrj\"anen, R.; Mart\'inez-Perales, J. C.; P\'erez, C.; V\"ah\"akangas, A.V. \emph{On the BBM-phenomenon in fractional Poincar\'e-Sobolev inequalities with weights}. Int. Math. Res. Not. IMRN 2023, no. 20, 17205--17244.  

\bibitem{HV} Hurri-Syrj\"anen, R.; V\"ah\"akangas, A. V. \emph{On fractional Poincaré inequalities}. J. Anal. Math. 120 (2013), 85--104. 

\bibitem{JKK} Jiang, R.;  Kauranen, A.;  Koskela, P. \emph{Solvability of the divergence equation implies John via Poincar\'e inequality}. Nonlinear Anal. Theory Methods Appl. 101 (2014), 80--88.

\bibitem{MPW} Myyryl\"ainen, K., P\'erez, C., Weigt, J. \emph{Weighted Fractional Poincar\'e inequalities via isoperimetric inequalities}. Preprint (2023). https://doi.org/10.48550/arXiv.2304.02681
 

\bibitem{MS} Maz'ya, V.; Shaposhnikova, T. \emph{On the Bourgain, Brezis, and Mironescu theorem concerning limiting embeddings of fractional Slobber spaces}. J. Funct. Anal. 195 (2002), no. 2, 230--238. 

\bibitem{M}  Milman, M. \emph{Notes on limits of Sobolev spaces and the continuity of interpolation scales}. Trans. Amer. Math. Soc. 357 (2005), no. 9, 3425--3442

\bibitem{Mo} Mohanta, K. \emph{Bourgan-Br\'ezis-Mironescu formula for Triebel-Lizorkin spaces in arbitrary domains}. arXiv:2308.12830v2 (2023).

\bibitem{Mu} Muckenhoupt, B. \emph{Weighted norm inequalities for the Hardy maximal function}. Trans. Amer. Math. Soc. 165 (1972), 207--226. 

\bibitem{LGO} L\'opez-Garc\'ia, F.; Ojea, I. \emph{Some inequalities on weighted Sobolev spaces, distance weights and the Assouad dimension}. Preprint (2022). 
https://doi.org/10.48550/arXiv.2210.12322

\bibitem{P} Ponce, A.C. \emph{ An estimate in the spirit of Poincar\'e’s inequality}. J. Eur. Math. Soc. (JEMS) 6 (2004), no. 1, 1--15.
 \end{thebibliography}
\end{document}